\documentclass[12pt,english]{amsart}

\input xy
\xyoption{all}

\usepackage[latin1]{inputenc}
\usepackage{amsfonts,amsmath,amssymb,amsthm}

\newcommand{\Pp}{\mathbb{P}}

\newcommand{\Nn}{\mathbb{N}}
\newcommand{\Zz}{\mathbb{Z}}
\newcommand{\Qq}{\mathbb{Q}} 
\newcommand{\Ff}{\mathbb{F}}

{\theoremstyle{plain}
\newtheorem{theorem}{Theorem}[section]    

\newtheorem{twisting lemma}[theorem]{Twisting lemma}

\newtheorem{lemma}[theorem]{Lemma}       
\newtheorem{proposition}[theorem]{Proposition}  
\newtheorem{corollary}[theorem]{Corollary}   

}
{\theoremstyle{remark}
\newtheorem{definition}[theorem]{Definition}      
\newtheorem{remark}[theorem]{Remark}   

}

\def\Gabs{\hbox{\rm G}}

\def\Gal{\hbox{\rm Gal}}
\def\max{\hbox{\rm max}}

\def\cm{\hbox{\hbox{\rm C}\kern-5pt{\raise 1pt\hbox{$|$}}}}

\def\lhfl#1#2{\smash{\mathop{\hbox to 12mm{\leftarrowfill}}
\limits^{#1}_{#2}}}

\def\rhfl#1#2{\smash{\mathop{\hbox to 12mm{\rightarrowfill}}
\limits^{#1}_{#2}}}

\def\build#1_#2^#3{\mathrel{
\mathop{\kern 0pt#1}\limits_{#2}^{#3}}}

\def\htrait#1#2{\smash{\mathop{\hbox to 12mm{\hrulefill}}
\limits^{#1}_{#2}}}

\def\sxbullet{{\raise 2pt\hbox{\bf .}}}

\begin{document}

\title[On the Malle conjecture and the self-twisted cover]{On the Malle conjecture and the self-twisted cover}

\author{Pierre D\`ebes}

\email{Pierre.Debes@math.univ-lille1.fr}

\address{Laboratoire Paul Painlev\'e, Math\'ematiques, Universit\'e Lille 1, 59655 Villeneuve d'Ascq Cedex, France}



\date{\today}

\begin{abstract} 
We show that for a large class of finite groups $G$, the number of Galois extensions $E/\Qq$ of group $G$ and discriminant $|d_E|\leq y$ grows like a power of $y$ (for some specified exponent). The groups $G$ 
are the regular Galois groups over $\Qq$ and the extensions $E/\Qq$ that we count are obtained by spe\-ci\-ali\-za\-tion from a given 
regular Galois extension $F/\Qq(T)$. The extensions $E/\Qq$ can further be prescri\-bed any unramified local behavior at each suitably large prime $p\leq \log (y)/\delta$ for some $\delta\geq 1$. 
This result is a step toward the Malle con\-jec\-tu\-re on the number of Galois extensions of given group and bounded discriminant. The local conditions fur\-ther make it a no\-ta\-ble constraint on regular Galois groups over $\Qq$.
 The method uses the notion of self-twisted cover that we introduce.
\end{abstract}

\maketitle


\section{Main results}\label{sec:main_results}

Given a finite group $G$ and a real number $y>0$, there are only finitely many  Galois extensions $E/\Qq$ (inside a fixed algebraic closure $\overline \Qq$ of $\Qq$) of group $G$ and discriminant $|d_E|\leq y$ (Hermite's theorem). Estimating their number $N(G,y)$ is a classical topic (\S \ref{ssec:malle}). 
Here we consider the extensions $E/\Qq$ obtained by specialization from a given Galois function field extension $F/\Qq(T)$ of group $G$ (\S \ref{ssec:contribution}). We obtain estimates for the number of those which satisfy the above group and ramification conditions. Our lower bound (obviously also a lower bound for $N(G,y)$) already has the conjectural growth for $N(G,y)$: a power of $y$ (\S \ref{ssec:main-statement}). Furthermore the extensions $E/\Qq$ we produce satisfy some additional local  conditions at a finite -- but growing with $y$ -- set of primes. This provides noteworthy constraints (though unknown yet not to be non-vacuous) on regular Galois groups over $\Qq$, related to analytic issues around the Tchebotarev theorem (\S \ref{ssec:grunwald-tchebotarev}). The role of the {\it self-twisted cover} from the title is explained in 
\S \ref{ssec:role-self-twisted}.


\subsection{The Malle conjecture \label{ssec:malle}{\rm is
a classical landmark in this context}} It predicts that for some constant $a(G) \in ]0,1]$, specifically defined by Malle (recalled in \S \ref{ssec:main-statement}), and for all $\varepsilon >0$, 
\vskip 1mm

\noindent
{\rm (*)} \hskip 5mm $c_1(G) \hskip 2pt y^{a(G)} \leq N(G,y) \leq c_2(G,\varepsilon) \hskip 2pt y^{a(G)+\varepsilon} \hskip 5mm \hbox{\rm for all $y>y_0(G,\varepsilon)$}$
\vskip 1,2mm
\noindent 
for some positive constants $c_1(G)$, $c_2(G,\varepsilon)$ and $y_0(G,\varepsilon)$  \cite{Ma}. A more precise asymptotic for $N(G,Y)$  (as $y\rightarrow \infty$) is even offered in \cite{Ma2}, namely $N(G,y) \sim c(G) \hskip 2pt y^{a(G)} \hskip 2pt  (\log(y))^{b(G)}$, for some other specified constant $b(G)\geq 0$, and an another (unspecified) constant $c(G)>0$.

The lower bound in (*) is a strong statement; it implies in particular that $G$ is the Galois group of at least one extension $E/\Qq$, 
which is an open question for many groups -- the so-called {\it Inverse Galois Problem}. Relying on the Shafarevich theorem solving the IGP for solvable groups, Kl\"uners and Malle proved the conjecture (*) for nilpotent groups \cite{kluners-malle}. 
Kl\"uners also established the lower bound for dihedral groups of order $2p$ with $p$ an odd prime \cite{kluners}.
As to the more precise asymptotic for $N(G,y)$, it has only been proved for abelian groups \cite{wright}, \cite{maki}, for $S_3$ \cite{belabas-fouvry} \cite{bhargava-wood} and for generalized quaternion groups \cite[ch.7, satz 7.6]{Klueners-habilitation}. 


We point out that there is a more general form of the conjecture for {not necessarily Galois extensions} $E/\Qq$ for which there are further significant results, notably in the case the group (of the Galois clo\-sure) is $S_n$ (with $n=[E:\Qq]$): results of Davenport-Heilbronn \cite{davenport-heilbronn} and Datskovsky-Wright
\cite{datskovsky-wright} ($n=3$), Bhargava \cite{barghava-S4}, \cite{barghava-S5} ($n=4,5$), Ellenberg-Venkatesh \cite{MR2199231} (upper bounds). There is also a counter-example to this more general form 
of the conjecture \cite{kluners2}. Finally there are quite interesting investigations on analogs  of the problem
over function fields of finite fields \cite{MR2159381}, \cite{MR2827801}, \cite{ellenberg-et-al}.

\subsection{Our specialization approach} \label{ssec:contribution} Besides solvable groups, there is another classical class of finite groups known to be Galois groups over $\Qq$: those groups $G$ which are {\it regular Galois groups} over $\Qq$, {\it i.e.}, such that there exists a Galois extension $F/\Qq(T)$ of group $G$ with $F\cap \overline \Qq = \Qq$. In addition to abelian and dihedral groups, this class includes many non solvable groups, {\it e.g.} all symmetric and alternating groups and many simple groups. We obtain for all these groups a lower bound in $y^\alpha$ for $N(G,y)$, as predicted by the Malle estimate (*). 

To our knowledge, this is a new step toward the conjecture. Expectedly our exponent $\alpha$ is smaller than its Malle counterpart $a(G)$: our approach only takes into account those extensions 
which are specializations of a geometric extension $F/\Qq(T)$.
Our result 
is rather a specialization result (as presented in \S \ref{sec1}). Still it is interesting to already get the right growth for $N(G,y)$ from a single extension $F/\Qq(T)$.


In this geometric situation, Hilbert's irreducibility theorem classically produces ``many'' $t_0 \in {\Qq}$ such that the corresponding specialized extensions $F_{t_0}/{\Qq}$ are Galois extensions of group $G$. Beyond making more precise these ``many $t_0 \in \Qq$'' and controlling the corresponding discriminants, our goal requires a further step which is to show that many of these extensions are distinct. It is for this part that the {\it self-twisted cover}, a novel tool that we 
construct in \S \ref{sec:self-twisted}, will be used (we say more in \S \ref{ssec:role-self-twisted}).
 



\subsection{Statement of the main result} \label{ssec:main-statement} 
In addition to being of group $G$ \hbox{and discriminant $\leq y$}, we will be able to prescribe their local behavior at many primes to the Galois extensions $E/\Qq$ that we will produce.
The following notation helps phrase these ``local conditions''.



Given a finite group $G$, a finite set $S$ of primes and for each $p\in S$, a subset ${\mathcal F}_p\subset G$ con\-si\-sting of a non-empty union of conjugacy classes of $G$, the collection 
${\mathcal F}= ({\mathcal F}_p)_{p\in S}$ is called a {\it Frobenius data} for $G$ 
on $S$.
The number of Galois extensions $E/\Qq$ of group $G$, of discriminant $|d_E| \leq y$ and which are unramified with Frobenius ${\rm Frob}_p(E/\Qq) \in {\mathcal F}_p$ ($p\in S$) is denoted by $N(G,y,{\mathcal F})$.

The parameter $\delta(G)$ that appears below in the definition of our exponent is the {\it minimal affine branching index} of regular realizations of $G$ over $\Qq$, 
{\it i.e.}, 
the minimal degree of the discriminant $\Delta_P(T)$ of a polynomial $P\in \Qq[T,Y]$, monic in $Y$,  
such that $\Qq(T)[Y]/\langle P\rangle$ is a regular Galois extension of $\Qq(T)$ of group $G$. 

\begin{theorem} \label{thm:main} Let $G$ be a regular Galois group over $\Qq$, non trivial. There exists a constant $p_0(G)$ with the following property. For every $\delta >\delta(G)$, for every suitably large $y$ (depending on $G$ and $\delta$) and every Frobenius data ${{\mathcal F}_y}$ on ${\mathcal S}_y=
\{p_0(G) < p \leq \log (y)/\delta\}$, we have
$$ N(G,y,{\mathcal F}_y) \geq \hskip 2pt y^{\alpha(G,\delta)}\hskip 8mm \hbox{\it with}\hskip 3mm \alpha(G,\delta)= (1-1/|G|)/\delta.$$ 

\noindent
Furthermore, the desired extensions $E/\Qq$ can be obtained by specializing some regular realization $F/\Qq(T)$ of $G$.
\end{theorem}

\S \ref{sec1} says more on $\delta(G)$. If a regular reali\-za\-tion $F/\Qq(T)$ of $G$ is gi\-ven by a polynomial $P\in \Qq[T,Y]$, monic in $Y$, then $\delta(G) < 2\hskip 1pt |G| \deg_T(P)$. One can then take $\delta=2 |G| \deg_T(P)$ in theorem \ref{thm:main} or the more intrinsic value $\delta = 3 \hskip 1pt r  \hskip 1pt |G|^3 \log|G|$  with $r$ the branch point number of $F/\Qq(T)$. By comparison, Malle's exponent is $a(G)= (|G| (1- 1/\ell))^{-1}$ where $\ell$ is the smallest prime divisor of $|G|$; inequality $a(G) \geq 1/\delta(G)$ is proved in general in lemma \ref{lem:exponent}, following a suggestion of G. Malle.



A more precise form of theorem \ref{thm:main} is stated in \S \ref{sec1} which starts from any given regular reali\-za\-tion $F/\Qq(T)$ of $G$ and which shows other features of our result: ramification can also be prescribed at any finitely many suitably large primes under the assumption that $F/\Qq(T)$ has at least one $\Qq$-rational branch point; 
the exponent $\alpha(G,\delta)$ can be replaced by $1/\delta$ under Lang's conjecture; the Hilbert irreducibility aspect is expanded; and there is an upper bound part; see theorem \ref{thm:main-plus}.

\begin{remark}  Extending theorem \ref{thm:main} (and theorem \ref{thm:main-plus}) to arbitrary number fields (instead of $\Qq$) seems to present no major obstacles, only
requiring to extend some ingredients we use, which are only available in the literature over $\Qq$ but should hold over number fields. As each finite group is known to be a regular Galois group over some suitably big number field, we could deduce that the same is true for the lower bound part in the Malle conjecture: {\it given any finite group, there is a number field $k_0$ such that a lower bound like in {\rm (*)} (appropriately generalized) holds over every number field containing $k_0$.} 
\end{remark}

\subsection{On the local conditions} \label{ssec:grunwald-tchebotarev}  (a) Regarding this aspect, theorem \ref{thm:main} im\-proves on our previous work, with N. Ghazi, about the {Grunwald problem}.
From \cite{DEGha}, if $G$ is a regular Galois group over $\Qq$, then every {\it unramified Grunwald problem} for $G$ at some finite set ${\mathcal S}$ of primes $p\geq p_0(G)$ can be solved, {\it i.e.} every collection of unramified extensions $E^p/\Qq_p$ of group $H_p\subset G$ ($p\in {\mathcal S}$) is induced by some Galois extension $E/\Qq$ of group $G$.  Theorem \ref{thm:main} does more: it provides, for every given discriminant size, a big number of such extensions $E/\Qq$, a number that grows as in Malle's predictions.
\vskip 1mm

Malle had suggested that his estimates should hold with some local conditions \cite[Remark 1.2]{Ma2}. However, unlike his, ours have a set of primes, ${\mathcal S}_y$, which grows with $y$. We focus below on this. 
\vskip 1mm

\noindent
(b) First we note that the set of primes where 
the local be\-ha\-vior can be prescribed as in theorem \ref{thm:main} cannot be expected to be much bigger than the set ${\mathcal S}_y$:
\vskip 0,5mm

\noindent
-  
indeed, that every possible Frobenius data on ${\mathcal S}_y$ occurs in at least one extension $E/\Qq$ counted by $N(G,y)$ already gives $N(G,y)\geq c^{u(y)}$, with $c$ the number of conjugacy classes of $G$ and $u(y)$ the number of primes in ${\mathcal S}_y=\{p_0(G) < p \leq \log (y)/\delta\}$. Now $c^{u(y)}$ compares to the conjectural upper bounds for $N(G,y)$: $\log c^{u(y)} \sim\log (y)/ \log(\log y)$ and $\log(y^{a(G)+\varepsilon}) \sim \log y$ (up to multiplicative constants). 
\vskip 0,5mm

\noindent
- the restriction that the primes $p$ be suitably large ($p>p_0(G)$) cannot be removed either as the famous Wang's counter-example shows \cite{wang}: no Galois extension $E/\Qq$ of group $\Zz/8\Zz$ is unramified at $2$ with Frobenius of order $8$. Other counter-examples with other primes than $2$ have been recently produced by Neftin \cite{Neftin}.
\vskip 1mm



\noindent
(c) There is a further connection of our result with the Tchebotarev density theorem. The following definition helps explain it. 

\begin{definition} \label{def:tche}
Given a real number $\ell \geq 0$, we say that a finite group $G$ is of {\it Tchebotarev exponent} $\leq \ell$, which we write ${\rm tch}(G)\leq \ell$, if there exist real numbers $m,\delta >0$ such that for every $x>m$ and every {\it Frobenius data} ${\mathcal F}_x= ({\mathcal F}_p)_{m<p\leq x}$ for $G$, there exists {\it at least one} Galois extension $E/\Qq$ of group $G$ such that these two conditions hold:

\noindent 1. for each $m<p\leq x$, $E/\Qq$ is unramified and ${\rm Frob}_p(E/\Qq) \in {\mathcal F}_p$,

\noindent 2. $\log |d_E| \leq \delta x^\ell$.
\end{definition}

Fix $\delta >\delta(G)$ and $m$ suitably large (in particular $m\geq p_0(G)$). 
Theorem \ref{thm:main} for $y=e^{\delta x}$ with $x>m$ provides many\footnote{at least $e^{\gamma x}$ with $\gamma=1-1/|G|$ (for $x\gg 1$).} extensions $E/\Qq$ satisfying  conditions of definition \ref{def:tche} with $\ell = 1$. 

\begin{corollary} \label{cor:tch(G)}
If a finite group $G$ is a regular Galois group over $\Qq$, then ${\rm tch}(G) \leq 1$.
\end{corollary}

On the other hand there is a universal lower bound for ${\rm tch}(G)$. Some famous estimates on the Tchebotarev theorem \cite{LaMoOd} (see also \cite{Lagarias-Odlyzko}, \cite{Se_Cebotarev}) show that, under the General Riemann Hypothesis, for every finite group $G$, we have 
$$ {\rm tch}(G) >  (1/2) - \varepsilon,\  \hbox{\rm for every}\ \varepsilon >0.\  \footnote{\cite{LaMoOd} also has an unconditional conclusion, which, using our terminology, leads to ${\rm tch}(G) \geq (\log \log x)/(2\log x)$ (with definition \ref{def:tche} extended to allow $\ell$ to be a function of $x$).} \leqno\hbox{\rm (**)}$$
\vskip 1mm
\noindent 
(More precisely, they show that if a Galois extension $E/\Qq$ is of group $G$ and $\log|d_E| \leq x^{1/2}/\log x$, there are at least $\pi(x) - 2x/(|G| \log x)$ non totally split primes $p\leq x$ in $E/\Qq$ (with $\pi(x)$ is the number of primes $p\leq x$). As $\pi(x) - 2x/(|G| \log x) \rightarrow +\infty$, the trivial totally split behavior --- ${\mathcal F}_p = \{1\}$ for each $m < p \leq x$ --- does not occur if $x\gg 1$).
\vskip 1mm

\hskip -1,5mm Corollary \ref{cor:tch(G)} raises the question of whether ${\rm tch}(G)>1$ for some group $G$, in which case $G$ could not be a regular Galois group over $\Qq$. Such a group may not exist (if the so-called Regular Inverse Galois Problem is true), while at the other extreme it cannot be ruled out at the moment that ${\rm tch}(G) =\infty$ for some group $G$. Many possibilities exist in between for Galois groups $G$ over $\Qq$: that realizations exist that satisfy the local conditions of definition \ref{def:tche}  (1) or not, that
the corresponding discriminants can be bounded as in definition \ref{def:tche} (2), for some $\ell \in  [1/2,\infty[$ or not. 
Somehow the Tchebotarev exponent provides a measure of the gap (possibly empty) between the {\it classical} and {\it regular} Inverse Galois Problems.

Gaining information on Tchebotarev exponents however seems difficult.
Even for $G=\Zz/2\Zz$ and the case of the totally split behavior, for which the problem amounts to bounding the least square-free 
integer $d_m(x)$ that is a quadratic residue modulo each prime $m<p\leq x$. Chan-ging $1/2$ to $1$ in (**), the remaining possible improvement (as $\Zz/2\Zz$ is a regular Galois group over $\Qq$), is plausible as some easy heuristics show 
but relates to deep questions in number theory ({\it e.g.} \cite[\S 2.5]{Se_Cebotarev}).

\subsection{Role of the self-twisted cover} \label{ssec:role-self-twisted}
Our method starts with a regular realization $F/\Qq(T)$ of $G$. The extensions $E/\Qq$ that we wish to produce will be specializations $F_{t_0}/\Qq$ of it at some integers $t_0$. A key tool is the twisting lemma from \cite{DEGha}, which reduces the search of specializations of a given type to that of rational points on a certain twisted cover. We use it twice, first over $\Qq_p$ as in \cite{DEGha}, to construct specializations $F_{t_0}/\Qq$ with a specified local behavior. A main ingredient for this first stage is the Lang-Weil estimate for the number of rational points on a curve over a finite field. We obtain many good specialisations $t_0\in \Zz$ and a lower bound for their number. 

The next question is to bound the number of the corresponding specializations $F_{t_0}/\Qq$ that are non-isomorphic. First we reduce it to counting integral points of a given size on 
certain twisted covers. This is our second use of the twisting lemma, over $\Qq$ this time. For the count of the integral points, we use a result of Walkowiak \cite{Wa}\footnote{We need to slightly improve Walkowiak's result to get the right exponent $\alpha(G,\delta)$ in theorem \ref{thm:main}; see \S \ref{ssec:walkowiak}.}, based on a method of Heath-Brown \cite{HB}. However the bounds from \cite{Wa} involve the height of the defining polynomials, which here depend on the specializations $F_{t_0}/\Qq$. We have to control the dependence in $t_0$. This is where enters the self-twisted cover, which as we will see, is a family of covers, depending only on the original extension $F/\Qq(T)$ and which has all the twisted covers among its fibers. As a result, a bound of the form $c_1 t_0^{c_2}$ for the height of the polynomials above will follow with $c_1$ and $c_2$ depending only on $F/\Qq(T)$.

\vskip 2mm 
In \S 2 below we present theorem \ref{thm:main-plus}, the more precise version of theorem \ref{thm:main}. \S 3 is the construction of the self-twisted cover. \S 4 gives the proof of theorem \ref{thm:main-plus}.

\vskip 2mm
\noindent 
{\bf Acknowledgement.} I am grateful to G. Malle and J. Kl\"uners for their interest in the paper and some valuable suggestions. 

\penalty -1300
\section{The specialization version of the result}\label{sec1}

\subsection{Basics}  \label{ssec:basics}
Given a field $k$, we work without distinction with a regular extension $F/k(T)$ or with the associated $k$-cover $f:X\rightarrow \Pp^1$: $f$ is the normalization of $\Pp^1_k$ in $F$ and $F$ is the function field $k(X)$ of $X$. 

We assume that $k$ is of characteristic $0$; $k=\Qq$ in most of the paper.

Recall that a {\it (regular) $k$-cover of $\Pp^1$} is a finite and generically unramified morphism $f:X \rightarrow \Pp^1$ defined over $k$ with $X$ a normal and geometrically irreducible variety. The $k$-cover $f:X \rightarrow \Pp^1$ is said to be {\it Galois} if the field extension $k(X)/k(T)$ is Galois; if in addition $f:X\rightarrow \Pp^1$ is given together with an isomorphism $G\rightarrow \Gal(k(X)/k(T))$, it is called a (regular) $k$-G-{\it Galois cover} of group $G$. 

By {\it group} and {\it branch point set} of a $k$-cover $f$, we mean 
those of the $\overline k$-cover  $f\otimes_k\overline k$: the group of a $\overline k$-cover $X\rightarrow \Pp^1$ is the  Galois group of the Galois closure of the extension $\overline k(X)/\overline k(T)$. 
The branch point set of $f\otimes_k\overline k$ is the (finite) set of points $t\in \Pp^1(\overline k)$ such that the associated discrete valuations are ramified in the extension $\overline k(X)/\overline k(T)$.

Given a regular Galois extension $F/k(T)$ and $t_0\in {\mathbb P}^1(k)$, the {\it specialization of $F/\Qq(T)$ at $t_0$} is the residue extension of an (arbitrary) prime above $\langle T-t_0\rangle$ in the integral closure of $\Qq[T]_{\langle T-t_0\rangle}$ in $F$ (as usual use $\Qq[1/T]_{\langle 1/T\rangle}$ instead if $t_0=\infty$). We denote it by $F_{t_0}/k$.

Given a regular Galois extension $F/k(T)$, we say a prime $p$ is {\it good} for $F/\Qq(T)$ if $p\not|\hskip 1mm |G|$, the branch divisor ${\bf t}=\{t_1,\ldots,t_r\}$ is \'etale at $p$ and there is no vertical ramification at $p$; and it is {\it bad} otherwise. We refer to \cite{DEGha} for the precise definitions. We only use here the standard fact that there are only finitely many bad primes.

\subsection{The minimal affine branching index $\delta(G)$}\label{ssec:exponent} 
Given a regular extension $F/\Qq[T]$, we call the irreducible polynomial $P(T,Y)$ of a primitive element, integral over $\Zz[T]$, an {\it integral affine model} of $F/\Qq(T)$;  $P(T,Y)\in \Zz[T]$ and is monic in $Y$. Denote the discriminant of $P$ relative to $Y$ by $\Delta_P(T)\in \Zz[T]$ and its degree by $\delta_P$. The minimal degree $\delta_P$ obtained in this manner is called the {\it minimal affine branching index of $F/\Qq(T)$} and denoted by $\delta_{F/\Qq(T)}$. For any integral affine model $P(T,Y)$ of $F/\Qq(T)$, we have
$$ \delta_{F/\Qq(T)} \leq \delta_P < 2 |G| \deg_T(P).$$

If $G$ is 
a regular Galois group over $\Qq$, the parameter $\delta(G)$ involved in theorem \ref{thm:main} is the minimum of all $\delta_{F/\Qq[T]}$ with $F/\Qq[T]$ running over all regular realizations of $G$.

\penalty -1200
\begin{lemma}\label{lem:exponent} Let $G$ be a non trivial regular Galois group over $\Qq$.
\vskip 1mm

\noindent
{\rm (a)} If $F/\Qq(T)$ is a regular realization of $G$ with $r$ branch points and $g$ is the genus of $F$, then
$$\delta(G) < 3 \hskip 1pt (2g+1) |G|^2 \log|G| \leq  3 \hskip 1pt r  \hskip 1pt |G|^3 \log|G|.$$
\vskip 1mm

\noindent
{\rm (b)} Furthermore we have \hskip 5mm  $\displaystyle \delta(G) \geq 1/a(G).$
\end{lemma}

\begin{proof}[Proof of lemma \ref{lem:exponent}] (a) The first inequality follows from a result of Sadi \cite[\S 2.2]{SadiT} which provides an affine model $P(T,Y)$ of $F/\Qq(T)$ such that 
$$ \deg_T(P) \leq (2g+1) |G| \log|G| /\log 2.$$ 
\noindent
The second inequality follows from the Riemann-Hurwitz formula.  
\vskip 1mm

(b) Let $F/\Qq(T)$ be a regular realization of $G$, $d_F\in \Qq[T]$ be the absolute discriminant of $F/\Qq(T)$ (the discriminant of a $\Qq[T]$-basis of the integral closure of $\Qq[T]$ in $F$) and $P(T,Y)$ be an integral affine model of $F/\Qq(T)$. Inequality (b) follows from  the following ones: 
$$\delta_P \geq \deg(d_F) \geq |G| (1- 1/\ell) = a(G).$$

\noindent
 where $\ell$ is as before the smallest prime divisor of $|G|$.

The first inequality $\deg(d_F) \leq \delta_P$ is standard.
Classically the polynomial $d_F$ is a generator of the ideal $N_{F/\Qq(T)}({\mathcal D}_{F/\Qq(T)})$ where ${\mathcal D}_{F/\Qq(T)}$ is the different and $N_{F/\Qq(T)}$ is the norm relative to the extension $F/\Qq(T)$ \cite[III, \S 3]{SeCorps_locaux}. From \cite[III, \S 6]{SeCorps_locaux}, in the prime ideal decomposition ${\mathcal D}_{F/\Qq(T)} =\prod_{\mathcal P} \mathcal P^{u_{\mathcal P}}$, we have $u_{\mathcal P}\geq e_{{\mathcal P}}-1$ for each prime ${\mathcal P}$, where $e_{\mathcal P} = e_{\frak p}$ is the corresponding ramification index, which only depends on the prime ${\frak p}$ below ${\mathcal P}$. The following inequalities, where $f_{{\mathcal P}}$ denotes the residue degree of ${\mathcal P}$, finish the proof:

\medskip
\centerline{$ \deg(d_F) \geq \sum_{{\frak p}} \sum_{{\mathcal P} | {\frak p}} f_{\mathcal P}(e_{\mathcal P} - 1) = \sum_{{\frak p}} |G| - |G|/e_{{\frak p}} \geq |G| (1-1/\ell)$}
%
%
%
%
\end{proof}

\begin{remark} Our parameter $\delta (G)$ can also be compared to the minimum, say $\rho(G)$, of all branch point numbers $r$ of regular realizations $F/\Qq(T)$ of $G$: for such an extension $F/\Qq(T)$ we have $\deg(d_F)\geq r-1$ whence 
$\delta(G) \geq \rho(G) - 1$.
But the inequality $a(G) \geq 1/(\rho(G)-1)$ does not hold in general. For example the symmetric group $S_n$ can be regularly realized over $\Qq$ with $3$ branch points so $\rho(S_n)= 3$ while $a(S_n) = 2/n!$ The analog of the\-orem \ref{thm:main} with $r-1$ replacing $\delta_{F/\Qq(T)}$ is false if the upper bound part of Malle's conjecture is true. 
\end{remark}

\subsection{The specialization result} \label{ssec:generalization} 

Theorem \ref{thm:main-plus} is a more precise version of our main result. It gives  
explicit estimates from which the asymptotic estimates of theorem \ref{thm:main} can easily be deduced (as explained in \S \ref{main-plus-implies-main}). Another difference is that it starts from a given regular Galois extension $F/\Qq(T)$ and the extensions $E/\Qq$ we count are specializations of it. Below are the necessary additional notation and data. 

\subsubsection{Notation} The following notation is used throughout the paper:
\vskip 1pt

\noindent
- for a Frobenius data ${{\mathcal F}}$ on a set of primes $S$, the product of all ratios $|{\mathcal F}_p|/{|G|}$ with $p\in S$, {\it i.e.} the {\it density} of ${{\mathcal F}}$, is denoted by $\chi({\mathcal F})$,
\vskip 1pt

\noindent
- for a finite set $S$ of primes, set $\Pi(S)=\prod_{p\in S}p$,  
\vskip 1pt

\noindent
- we also use the classical functions $\pi(x)$ and ${\Pi}(x)$ to denote respectively the number of primes $\leq x$ and the product of all primes $\leq x$. We have the classical asymptotics at $\infty$: $\pi(x) \sim x/\log(x)$ and $\log \Pi(x) \sim x$, 
\vskip 1pt

\noindent
- the height of a polynomial $F$ with coefficients in $\Qq$ is the maximum of the absolute values of its coefficients and is denoted by $H(F)$.

\subsubsection{Data} \label{ssec:data} Fix the following for the rest of the paper:
\vskip 1pt

\noindent
- $G$ is a non trivial finite group,
\vskip 1pt


\noindent
- $F/\Qq(T)$ is a regular Galois extension of group $G$,
\vskip 1pt
\noindent
- $f:X\rightarrow \Pp^1$ is the corresponding $\Qq$-cover,
\vskip 1pt

\noindent
- ${\mathbf t} = \{t_1,\ldots,t_r\} \subset \Pp^1(\overline \Qq)$ is the branch point set of $F/\Qq(T)$,
\vskip 1pt


\noindent
- $g$ is the genus of the curve $X$, 
\vskip 1pt

\noindent
- $p_0(F/\Qq(T))$ is the prime defined as follows. Let $p_{-1}$ be the biggest prime $p$ such $p$ is bad for $F/\Qq(T)$ or  $p< r^2 |G|^2$.  Then $p_0(F/\Qq(T))$ is the smallest prime $p$ such that the interval $]p_{-1},p_0]$ contains as many primes as there are non-trivial conjugacy classes of $G$. For the rest of the paper we fix a prime $p_0 \geq p_0(F/\Qq(T))$.

\vskip 2pt

\noindent
- $\delta_{F/\Qq(T)}$ or $\delta_F$ for short is the minimal affine branching index of $F/\Qq(T)$, 
\vskip 3pt

\noindent
- $P(T,Y)$ is an  integral affine model of $F/\Qq(T)$ such that $\delta_P=\delta_F$ (with $\delta_P$ the degree of the discriminant $\Delta_P$) 
and which is primitive, {\it i.e.} has relatively prime coefficients (an assumption which one can always reduce to),

\vskip 2pt
\noindent
- $S$ is a finite set of primes subject to these conditions:

\noindent
(a)  if no branch point of $f$ is in $\Zz$ then $S=\emptyset$.

\noindent
(b)  if at least one of the branch points of $f$, say $t_1$ is in $\Zz$, then $S$ is a finite set of good primes $p$, not dividing $t_1$ and not in $]p_{-1},p_0]$.
\vskip 1mm

\noindent
(If at least one branch point is $\Qq$-rational, one can reduce to the assumption in (b) {\it via} a simple change of the variable $T$).
\vskip 1mm
\noindent
- for $x> p_0$, $S_x$  is the set of primes $p$ such that $p_0<p\leq x$ and $p\notin S$.
\vskip 1mm

\noindent
- for technical reasons we change the condition ``$|d_E| \leq y$'' from theorem \ref{thm:main} to the more complicated one ``$|d_E| \leq \rho(x)$'' where 
$$\rho(x) = (1+\delta_P) H(\Delta_P) [\Pi(S) \hskip 1pt \Pi(x)]^{\delta_P}$$
\noindent
\vskip -1mm
\noindent (we have $\log \rho(x) \sim \delta_P x$ and so a simple change of variable leads back to the original condition).
\vskip 1mm

\subsubsection{Statement} \label{ssec:statement}
For $x> p_0$ let ${\mathcal F}_x$ be a Frobenius data on ${S}_x$. Theorem \ref{thm:main-plus} is about the number $N_F(x,S,{\mathcal F}_x)$ of distinct specializations $F_{t_0}/\Qq$ at points $t_0\in \Zz$ that satisfy
\vskip 1mm

\noindent
{\rm (i)} $\Gal(F_{t_0}/\Qq) = G$,

\noindent
{\rm (ii)} for each $p\in {S}_x$, $F_{t_0}/\Qq$ is unramified and ${\rm Frob}_p(F_{t_0}/\Qq) \in {\mathcal F}_p$,

\noindent
{\rm (iii)} for each $p\in S$, $F_{t_0}/\Qq$ is ramified at $p$,

\noindent
{\rm (iv)} $|d_{F_{t_0}}| \leq \rho(x)$.


\begin{theorem} \label{thm:main-plus} 

{\rm (a)} There exist constants $C_1,C_2,C_3,C_4$ only depending on $P(T,Y)$ such that for every $x>p_0$, we have
$$N_F(x,S,{\mathcal F}_x) \geq  \hskip 2pt C_{1} \hskip 2pt \frac{\chi({\mathcal F}_x)}{\Pi(S)^2} \hskip 4pt \frac{{\Pi}(x)^{1-1/|G|}}{(\log {\Pi}(x))^{C_2} \hskip 2pt C_3^{\pi(x)}} - C_4
$$
\noindent
\vskip 2mm
\noindent
so \hskip 7mm $\log(N_F(x,S,{\mathcal F}_x))$ is bigger than a function $\lambda(x) \sim (1-1/|G|)\hskip 1pt x$.
\vskip 1mm


%
%
%
%
%
\vskip 1,5mm
\noindent
{\rm (b)} Furthermore the specializations $F_{t_0}/\Qq$ counted by the lower bound can be taken to be specializations at integers $t_0 \in [1, \Pi({S})\hskip 2pt \Pi(x)]$. 
\vskip 1,5mm

\noindent
{\rm (c)} Under Lang's conjecture on rational points on a variety of general ty\-pe and if $g\geq 2$, we have this better inequality
$$N_F(x,S,{\mathcal F}_x) \geq  \hskip 2pt C_{5} \hskip 2pt \frac{\chi({\mathcal F}_x)}{\Pi(S)^2} \hskip 4pt \frac{{\Pi}(x)}{C_6^{\pi(x)}} - C_7
$$
\noindent
for constants $C_5,C_6,C_7$ only depending on $P$, so then
\noindent
\vskip 1mm
\noindent
\hskip 12mm $\log(N_F(x,S,{\mathcal F}_x))$ is bigger than a function $\lambda(x) \sim x$.
\vskip 1,5mm 

\noindent
{\rm (d)} We have the following upper bound for the number ${\mathcal N}_F(x,{\mathcal F}_x)$ of integers $t_0 \in [1,\Pi({S_x)}]$ such that condition {\rm (ii)} above holds:
$${\mathcal N}_F(x,{\mathcal F}_x)  \leq \hskip 2pt \chi({\mathcal F}_x) \hskip 1mm \frac{\Pi({S_x)}}{\beta} \hskip 1mm  \left( 2 - \lambda\right)^{|S_x|}$$
\noindent
\hskip 0mm where $\lambda = (r|G|-1)/r^2 |G|^2\in ]0,1/4]$ and $\beta$ depends only on $F/\Qq(T)$.
\end{theorem}

\subsubsection{Proof of theorem \ref{thm:main} assuming theorem \ref{thm:main-plus}}
\label{main-plus-implies-main}
Pick a regular realization $F/\Qq(T)$ of $G$ and an integral affine model $P(T,Y)$ such that $\delta_P=\delta_F=\delta(G)$. Set $p_0(G) = p_0(F/\Qq(T))$. Fix $\delta > \delta(G)$ and set $\delta^- = (\delta+\delta(G))/2$. Let $y>0$ and $x= \log (y)/\delta^-$.

As $\delta^-<\delta$ we have ${\mathcal S}_y \subset {S}_x$. Complete the given Frobenius data ${\mathcal F}_y$ on ${\mathcal S}_y$ in an arbitrary way to make it a Frobenius data ${\mathcal F}_x$ on $S_x$.  Apply theorem \ref{thm:main-plus} to ${\mathcal F}_x$ and $S=\emptyset$. As $\log \rho(x)\sim (\delta_P/\delta^-) \log y$, if $y$ is suitably large, $\rho(x) \leq y$. It follows that $N(G,y,{\mathcal F}_y) \geq N_F(x,\emptyset ,{\mathcal F}_x)$
and so $N(G,y,{\mathcal F}_y)$ can be bounded from below by the right-hand side term of the inequality from theorem \ref{thm:main-plus} (a) with  $x= \log (y)/\delta^-$. The logarithm of this term is asymptotic to $(1-1/|G|) \log (y) /\delta^-$. Conclude that for suitably large $y$, this term is bigger than $y^{(1-1/|G|)/\delta}$. 

\subsubsection{Remarks}
(a) In the situation $S\not=\emptyset$, for which it is possible to prescribe ramification at some primes, the assumption that at least one branch point is $\Qq$-rational cannot be removed, as explained in Legrand's paper \cite{Legr1}, which is devoted to this situation. Many groups  have a regular realization $F/\Qq(T)$ satisfying this assumption, although being of even order is a necessary condition \cite[\S 3.2]{Legr1}: abelian groups of even order, symmetric groups $S_n$ ($n\geq 2$), alternating groups $A_n$ ($n\geq 4$), many simple groups (including the Monster), etc. 

For these groups, theorem \ref{thm:main-plus} leads to a generalized version of the inequality from theorem \ref{thm:main} where  the left-hand term $N(G,y,{\mathcal F}_y)$ is replaced by the (smaller) number, say $N(G,S,y,{\mathcal F}_y)$, of extensions $E/\Qq$ which, in addition to the conditions prescribed in theorem \ref{thm:main}, are required to ramify at every prime from a finite set $S$ of suitably big primes (and where the set ${\mathcal S}_y$ is of course replaced by ${\mathcal S}_y\setminus S$).
%
%
\vskip 1mm

(b) The upper bound in theorem \ref{thm:main-plus} (d) concerns extensions $E/\Qq$ that are specializations of a given regular extension $F/\Qq(T)$ (at integers $t_0$) and so does not directly lead to upper bounds for  $N(G,y,{\mathcal F}_y)$. A natural hypothesis to make in this context is that
$G$ has a generic extension $F/\Qq(T)$ (or more generally a parametric extension, as defined in \cite{Legr1}): indeed then all Galois extensions $E/\Qq$ of group $G$ are specializations of $F/\Qq(T)$ (at points $t_0\in \Qq$). But only the four groups $\{1\}$, $\Zz/2\Zz$, $\Zz/3\Zz$, $S_3$ have a generic extension $F/\Qq(T)$ \cite[p.194]{JLY}.

\subsubsection{An application of the upper bound part from theorem \ref{thm:main-plus}} For every $x > p_0$, let ${\mathcal N}_{\rm tot.split}(x)$ be the set of all integers $t_0\geq 1$ such that the specialization $F_{t_0}/\Qq$ is totally split at each prime $p_0 <p\leq x$,  {\it i.e.} which satisfy condition (ii) from theorem \ref{thm:main-plus} with ${\mathcal F}_p$ taken to be the trivial conjugacy class for each $p\in S_x$.  

%
\begin{corollary}
For every $x > p_0$, ${\mathcal N}_{\rm tot.split}(x)$ is a union of (many) cosets modulo $\Pi({S_x})$ 
but its density decreases to $0$ as $x\rightarrow +\infty$.
\end{corollary}

\begin{proof}
We anticipate on \S \ref{sec:proof} to say that ${\mathcal N}_{\rm tot.split}(x)$ is a union of cosets modulo $\Pi({S_x})$ 
(see proposition \ref{prop:first-part} (c))
and focus on the density part of the statement. Every integer $t_0\in {\mathcal N}_{\rm tot.split}(x)$  writes
$t_0 = u + k \hskip 2pt \Pi({S_x})$ with $u$ one of the elements in $[1,\Pi({S_x})]$ counted by  ${\mathcal N}_F(x,{\mathcal F}_x)$ and $k\in \Zz$. Let $N\geq 1$ be any integer. If $1\leq t_0 \leq N$, then $k \leq N/\Pi({S_x})$. It follows then from theorem \ref{thm:main-plus} (d) that 
$$ \left|{\mathcal N}_{\rm tot.split}(x) \cap [1,N]\right| \leq \frac{N}{\Pi({S_x})} \times {\mathcal N}_F(x,{\mathcal F}_x) \leq  \hskip 2pt \frac{N}{\beta} \times  \left( \frac{2-\lambda}{|G|}\right)^{|S_x|}$$

\noindent
which divided by $N$ tends to $0$ as $x\rightarrow +\infty$. \end{proof}

%

Similar density conclusions can be obtained for other local behaviors for which the sets ${\mathcal F}_p$ are not too big compared to $G$.

\section{The self-twisted cover} \label{sec:self-twisted}

In \S \ref{sec:twisting}, we recall the twisting operation on covers and the twisting lemma (\S \ref{ssec:twisting_lemma}) and explain the motivation for introducing the self-twisted cover (\S \ref{ssec:motivation}). \S \ref{ssec:self-twisted} is devoted to its construction. Covers are viewed here as fundamental group representations. The correspondence is briefly recalled in \S \ref{ssec:fund-group-basics}.

\subsection{Twisting G-Galois covers} \label{sec:twisting}

For the material of this subsection, we refer to \cite{DEGha}.

\subsubsection{Fundamental groups representations of covers} \label{ssec:fund-group-basics}
Given a field $k$, denote its absolute Galois group by $\Gabs_k$.

If $E/k$ is a Galois extension of group $G$, an epimorphism $\varphi: \Gabs_k\rightarrow G$ such that $E$ is the fixed field of ${\rm ker}(\varphi)$ in $\overline k$ 
is a called a {\it $\Gabs_k$-representation} of $E/k$.

Given a finite subset ${\mathbf t}\subset \Pp^1(\overline k)$ invariant under $\Gabs_k$, the $k$-fundamental group of $\Pp^1\setminus {\mathbf t}$ is denoted by $\pi_1(\Pp^1\setminus {\mathbf t}, t)_k$; here $t$ denotes the fixed {\it base point}, which corresponds to choosing an embedding of $k(T)$ in an algebraically closed field $\Omega$. The $\overline k$-fundamental group $\pi_1(\Pp^1\setminus {\mathbf t}, t)_{\overline k}$ is defined as the Galois group of the maximal algebraic extension $\Omega_{{\mathbf t},k}/\overline k(T)$ (inside $\Omega$) unramified above $\Pp^1\setminus {\mathbf t}$ and the $k$-fundamental group $\pi_1(\Pp^1\setminus {\mathbf t}, t)_{k}$ as the group of the Galois extension $\Omega_{{\mathbf t},k}/k(T)$. 

Degree $d$ $k$-covers of $\Pp^1$  (resp. $k$-G-Galois covers of $\Pp^1$ of group $G$) with branch points in $\mathbf t$ correspond to transitive homomorphisms $\pi_1(\Pp^1\setminus {\mathbf t}, t)_{k} \rightarrow S_d$ (resp. to epimorphisms $\pi_1(\Pp^1\setminus {\mathbf t})_{k},t) \rightarrow G$), with the extra regularity condition that the restriction of $\phi$ to $\pi_1(B\setminus D, t)_{\overline k}$ remains transitive (resp. remains onto). These corresponding homomorphisms are called the {\it fundamental group representations} (or {\it $\pi_1$-representations} for short) of the cover $f$ (resp the G-cover $f$).

Each $k$-rational point $t_0\in \Pp^1(k)\setminus {\mathbf t}$ provides a section
${\sf s}_{t_0}: \Gabs_k\rightarrow \pi_1(\Pp^1\setminus {\mathbf t},t)_{k}$ to the exact sequence
$$ 1\rightarrow \pi_1(\Pp^1\setminus {\mathbf t}, t)_{\overline k} \rightarrow \pi_1(\Pp^1\setminus {\mathbf t},t)_k \rightarrow \Gabs_k \rightarrow 1$$ 
\noindent
which is uniquely defined up to conjugation by an element in the fundamental group $\pi_1(\Pp^1\setminus {\mathbf t}, t)_{\overline k}$.

 If $\phi: \pi_1(\Pp^1\setminus {\mathbf t}, t)_k \rightarrow G$ represents a $k$-G-Galois cover $f:X\rightarrow \Pp^1$, the morphism $\phi \circ {\sf s}_{t_0}:\Gabs_k \rightarrow G$ is the {\it specialization representation} of $\phi$ at $t_0$. The fixed field in $\overline k$ of ${\rm ker}(\phi \circ {\sf s}_{t_0})$ is the specialization $k(X)_{t_0}/k(T)$ of $k(X)/k(T)$
 at $t_0$ (defined in \S \ref{ssec:basics}). 

\subsubsection{The twisting lemma} \label{ssec:twisting_lemma}
We recall how a regular $k$-G-Galois cover $f:X\rightarrow \Pp^1$ of group $G$ can be twisted by some 
Galois extension $E/k$ of group $H\subset G$. Formally this is done in terms of the associated $\pi_1$- 
and $\Gabs_k$- representations. 

Let $\phi: \pi_1(\Pp^1\setminus {\mathbf t}, t)_k \rightarrow G$ be a $\pi_1$-representation of the regular $k$-G-cover $f:X\rightarrow \Pp^1$ and $\varphi:\Gabs_k \rightarrow G$ be a $\Gabs_k$-representation of  the Galois extension $E/k$.

Denote the right-regular (resp. left-regular) representation of $G$ by $\delta: G\rightarrow S_d$ (resp. by $\gamma: G\rightarrow S_d$) where $d=|G|$. Define $\varphi^\ast:\Gabs_k \rightarrow G$ by
$\varphi^\ast (g) = \varphi (g) ^{-1}$. Consider the map $\widetilde \phi^\varphi:
\pi_1(\Pp^1\setminus {\mathbf t}, t)_k \rightarrow S_d$ defined by the following formula, where $r$ is the restriction map $\pi_1(\Pp^1\setminus {\mathbf t}, t)_k \rightarrow \Gabs_k$ and $\times$ is the multiplication in the symmetric group $S_d$:
$$ \widetilde \phi^\varphi(\theta)  =
\gamma \phi(\theta) \times \delta \varphi^\ast r (\theta) \hskip 6mm (\theta\in \pi_1(\Pp^1\setminus {\mathbf t}, t)_k). $$
\noindent
The map $\widetilde \phi^\varphi$ is a group homomorphism
with the same restriction on $\pi_1(\Pp^1\setminus {\mathbf t}, t)_{\overline k}$
as $\phi$. It is called  the {\it twisted representation} of 
$\phi$ by $\varphi$. The associated regular $k$-cover is a $k$-model 
of the cover $f\otimes_k{\overline k}$. It is denoted by $\widetilde f^\varphi: 
\widetilde X^\varphi \rightarrow \Pp^1$ and called  the {\it twisted cover} of 
$f$ by $\varphi$. The following statement is the main property 
of the twisted cover.

\begin{twisting lemma} \label{prop:twisted cover} 
Let $t_0\in \Pp^1(k)\setminus {\mathbf t}$. Then the specialization representation $\phi \circ {\sf s}_{t_0}:\Gabs_k \rightarrow G$ 
is conjugate in $G$ to $\varphi:\Gabs_k \rightarrow G$ if and only if there exists $x_0\in \widetilde X^\varphi(k)$ such that $\widetilde f^\varphi(x_0)=t_0$.
\end{twisting lemma}

\subsubsection{The motivation for the self-twisted cover} \label{ssec:motivation}
As explained in \S \ref{ssec:role-self-twisted}, we will have to control the height of some polynomials defining some twisted covers. These twisted covers are obtained by twisting the given G-Galois cover $f:X\rightarrow \Pp^1$ by its own specializations $\Qq(X)_{u_0}/\Qq$ ($u_0\in \Qq$); we call them the {\it fiber-twisted covers}. \S \ref{ssec:self-twisted} shows that the fiber-twisted covers are all members of an algebraic family of covers: the {\it self-twisted cover}. The practical use for the end of the paper is the following result. It is a consequence of lemma \ref{lemma:spec-self-twist}.


\begin{theorem} \label{prop:self-twist}
Given a regular $k$-G-cover $f:X\rightarrow \Pp^1$, there exists a polynomial $\widetilde P(U,T,Y) \in k[U,T,Y]$ irreducible in $\overline{k(U)}(T)[Y]$, monic in $Y$, and a finite set ${\mathcal E}\subset k$ such that for every $u_0\in k\setminus {\mathcal E}$, 

\noindent
{\rm (a)} $\widetilde P(u_0,T,Y)$ is irreducible in $\overline{k}(T)[Y]$,

\noindent
{\rm (b)} $\widetilde P(u_0,T,Y)$ is an integral affine model of the fiber-twisted cover of $f$ at $u_0$,

\noindent
{\rm (c)} the genus of the curve $\widetilde P(u_0,t,y)=0$ is $g$.
\end{theorem}

\subsection{The self-twisted cover} \label{ssec:self-twisted}
Let $U$ be a new indeterminate (algebraically independent from $T$ and $Y$).  
%
Fix an algebraically closed field $\Omega$ containing $k(T,U)$, which we will use as a common base point $t$ for all fundamental groups involved. The algebraic closures of $k(T,U)$, $k(T)$, $k(U)$ and $k$ should be understood as the ones inside $\Omega$.

\subsubsection{A $\pi_1$-representation of $f\otimes_kk(U)$} As the compositum $\Omega_{{\mathbf t},k} \cdot\overline{k(U)}$ is contained in $\Omega_{{\mathbf t},k(U)}$, there is a restriction morphism
$${\rm res}_{k(U)/k}: \pi_1(\Pp^1\setminus {\mathbf t}, t)_{{k(U)}} \rightarrow \pi_1(\Pp^1\setminus {\mathbf t}, t)_{k}, $$
\noindent
which induces a map between the geometric parts of the fundamental groups:
$${\rm res}_{\overline{k(U)}/\overline k}: \pi_1(\Pp^1\setminus {\mathbf t}, t)_{\overline{k(U)}}\rightarrow \pi_1(\Pp^1\setminus {\mathbf t}, t)_{\overline k}$$
\noindent 
We also use the notation ${\rm res}_{\overline{k(U)}/\overline k}$  for the map $\Gabs_{k(U)} \rightarrow \Gabs_k$ induced on the absolute Galois groups.

\begin{lemma}
${\rm res}_{k(U)/k}: \pi_1(\Pp^1\setminus {\mathbf t}, t)_{{k(U)}} \rightarrow \pi_1(\Pp^1\setminus {\mathbf t}, t)_{k}$ is surjective and 
${\rm res}_{\overline{k(U)}/\overline k}: \pi_1(\Pp^1\setminus {\mathbf t}, t)_{\overline{k(U)}}\rightarrow \pi_1(\Pp^1\setminus {\mathbf t}, t)_{\overline k}$ is an isomorphism.
\end{lemma}

\begin{proof} Every $\sigma \in \pi_1(\Pp^1\setminus {\mathbf t}, t)_{k}$ extends to an element of $\Gabs_{k(T)}$, which extends naturally to an automorphism of $\overline{k(T)}(U)$ fixing $U$ (and $k(T)$), which in turn extends to an element $\tilde \sigma \in \Gabs_{{k}(T,U)}$. As ${\mathbf t}$ is $\Gabs_k$-invariant, $\tilde \sigma$ permutes the extensions $F/{k(U)}(T)$ that are unramified above $\Pp^1\setminus \mathbf t$. Conclude that $\tilde \sigma$ factors through $\pi_1(\Pp^1\setminus {\mathbf t}, t)_{{k(U)}}$ to provide a preimage of $\sigma$ {\it via} the map ${\rm res}_{k(U)/k}$, as desired in the first statement. 

To show that ${\rm res}_{\overline{k(U)}/\overline k}:\pi_1(\Pp^1\setminus {\mathbf t}, t)_{\overline{k(U)}}\rightarrow \pi_1(\Pp^1\setminus {\mathbf t}, t)_{\overline k}$ is surjective, it suffices to show that the following morphism is:
$$\Gal(\Omega_{{\mathbf t},k} \cdot \overline{k(U)} / \overline{k(U)}(T)) \rightarrow \pi_1(\Pp^1\setminus {\mathbf t}, t)_{\overline k}=\Gal(\Omega_{{\mathbf t},k} / \overline{k}(T)).$$
\noindent
This morphism is in fact an isomorphism: indeed extending the base field from $\overline k$ to $\overline{k(U)}$ 
(over which $T$ is transcendental) does not change the group of regular Galois extensions.

As $k$ is of characteristic $0$, the morphism ${\rm res}_{\overline{k(U)}/\overline k}: \pi_1(\Pp^1\setminus {\mathbf t}, t)_{\overline{k(U)}}\rightarrow \pi_1(\Pp^1\setminus {\mathbf t}, t)_{\overline k}$ is even an isomorphism. More precisely, it follows from \cite[theorem 6.3.3]{Serre-topics} that $\Omega_{{\mathbf t},k(U)}= \Omega_{{\mathbf t},k} \cdot \overline{k(U)}$.
\end{proof}

Set $\phi \otimes_k k(U) = \phi \circ {\rm res}_{{k(U)}/k}$. The epimorphism
$$\phi \otimes_k k(U): \pi_1(\Pp^1\setminus {\mathbf t}, t)_{k(U)} \rightarrow G$$
\noindent
is a $\pi_1$-representation of the regular G-Galois cover $f\otimes_kk(U)$.

\subsubsection{A $\Gabs_{k(U)}$-representation} Composing $\phi \otimes_k k(U)$ with the section ${\sf s}_{U}: \Gabs_{k(U)} \rightarrow \pi_1(\Pp^1\setminus {\mathbf t}, t)_{k(U)}$ associated with the point $U\in \Pp^1(k(U))$ provides a $\Gabs_{k(U)}$-representation
$$\phi_{U}: \Gabs_{k(U)} \rightarrow G$$
\noindent 
which is the specialization representation of $\phi \otimes_k k(U)$ at $t=U$. It corresponds to the generic fiber of $F/k(T)$. Denote it by $F_U/k(U)$.

%
%
%
%

\subsubsection{The self-twisted cover} 
Twist the representation $\phi \otimes_k k(U)$ by the epimorphism $\phi_U$ to get the {\it self-twisted representation}
$$\widetilde{\phi \otimes_k k(U)}^{\phi_U}:  \pi_1(\Pp^1\setminus {\mathbf t}, t)_{k(U)} \rightarrow S_d.$$
\noindent
We call the corresponding cover
$$\widetilde{f\otimes_k k(U)}^{\phi_U}: \widetilde{X \otimes_k k(U)}^{\phi_U}\rightarrow \Pp^1_{k(U)}$$
\noindent
the {\it self-twisted cover} of $f$.

\subsubsection{The fiber-twisted cover at $t_0$} 
Let $t_0\in \Pp^1(k)\setminus{\mathbf t}$. 
Twist the representation $\phi$ by the specialization representation  $\phi \circ {\sf s}_{t_0}:\Gabs_k \rightarrow G$ to get the twisted representation
$$\widetilde{\phi}^{\phi \hskip 1pt {\sf s}_{t_0}}: \pi_1(\Pp^1\setminus {\mathbf t}, t)_{k} \rightarrow S_d$$
\noindent
which corresponds to a cover
$$\widetilde{f}^{\phi \hskip 1pt {\sf s}_{t_0}}: \widetilde{X}^{\phi \hskip 1pt {\sf s}_{t_0}}\rightarrow \Pp^1_{k}.$$
\noindent
We call them respectively the {\it fiber-twisted representation} and the {\it fiber-twisted cover} at $t_0$.

%
%
%
%

\subsubsection{Description of the self-twisted cover}
Set $\Psi_U= \widetilde{\phi \otimes_k k(U)}^{\phi_U}$. For every $\Theta\in \pi_1(\Pp^1\setminus {\mathbf t}, t)_{k(U)}$, we have
$$\Psi_U (\theta)  = \gamma(((\phi \otimes_k k(U)) (\Theta)) \times \delta (\phi_U(R(\Theta))^{-1})$$
\noindent
where $R: \pi_1(\Pp^1\setminus {\mathbf t}, t)_{k(U)} \rightarrow \Gabs_{k(U)}$ is the natural surjection. The element $\Theta$ uniquely writes $ \Theta = \chi \hskip 2pt {\sf s}_U(\sigma)$ with $\chi \in \pi_1(\Pp^1\setminus {\mathbf t}, t)_{\overline{k(U)}}$ and $\sigma\in \Gabs_{k(U)}$.
Whence
$$(\phi \otimes_k k(U)) (\Theta)) = (\phi \otimes_k k(U)) (\chi) \hskip 2pt (\phi \otimes_k k(U)) ({\sf s}_U(\sigma))$$
\noindent
and, using that $\phi_U = \phi \otimes_k k(U) \circ {\sf s}_{U}$,
$$\phi_U(R(\Theta)) = (\phi \otimes_k k(U))({\sf s}_{U}(\sigma)).$$
\noindent
Finally we obtain the following formula, where, by ${\rm conj}(g)$  ($g\in G$),  we denote the permutation of $G$ induced by the conjugation $x\rightarrow gxg^{-1}$:
$$\Psi_U (\theta) = \gamma((\phi \otimes_k k(U)) (\chi)) \times {\rm conj}({(\phi \otimes_k k(U))({\sf s}_{U}(\sigma))}).$$
Denote the field extension corresponding to the $\pi_1$-representation $\Psi_U$ by $\widetilde{F{k(U)}}^{\phi_U}/k(U)(T)$. The field $\widetilde{F{k(U)}}^{\phi_U}$ is the fixed field in $\Omega_{{\mathbf t},k(U)}$ of the subgroup $\Gamma_U\subset \pi_1(\Pp^1\setminus {\mathbf t}, t)_{{k(U)}}$ of all elements $\Theta$ such that $\Psi_U (\theta)$ fixes the neutral element of $G$ \footnote{Taking any other element of $G$ gives the same field up to $k(U)(T)$-isomorphism.}.
We obtain 
$$ \Gamma_U = {\rm ker}(\phi \otimes_k k(U)) \cdot {\sf s}_U(\Gabs_{k(U)})$$
\noindent
and $\widetilde{F{k(U)}}^{\phi_U}$ is the fixed field in $F\hskip 1pt k(U)$ of all elements in ${\sf s}_U(\Gabs_{k(U)})$.

\subsubsection{Description of the fiber-twisted covers}
Let $t_0\in \Pp^1(k)\setminus{\mathbf t}$ and set $\phi_{t_0}= \phi \circ {\sf s}_{t_0}$ and $\Psi_{t_0} = \widetilde{\phi}^{\phi_{t_0}}$.
Every element $\theta\in \pi_1(\Pp^1\setminus {\mathbf t}, t)_{k}$ uniquely writes $\theta = x\hskip 2pt {\sf s}_{t_0}(\tau)$ with $x \in \pi_1(\Pp^1\setminus {\mathbf t}, t)_{\overline{k}}$ and $\tau \in \Gabs_{k}$. Proceeding exactly as above but with $U$ replaced by $t_0$, $\phi \otimes_k k(U)$ by $\phi$ and $\Theta = \chi \hskip 2pt {\sf s}_U(\sigma)$ by $\theta = x\hskip 2pt {\sf s}_{t_0}(\tau)$, we obtain that
$$\Psi_{t_0} (\theta) = \gamma(\phi(x)) \times {\rm conj}({\phi({\sf s}_{t_0}(\tau))}) $$
\noindent
and if $\widetilde F^{\phi_{t_0}}/k(T)$ is the field extension corresponding to the $\pi_1$-repre\-sen\-ta\-tion $\Psi_{t_0}$, $\widetilde F^{\phi_{t_0}}$ is the fixed field in $F$ of all elements in ${\sf s}_{t_0}(\Gabs_{k})$.

\subsubsection{Comparison} 

\begin{lemma} \label{lemma:spec-self-twist}
There is a finite subset ${\mathcal E}\subset k$ such that for each $t_0\in k\setminus {\mathcal E}$, the fiber-twisted cover $\widetilde{f}^{\phi \hskip 1pt {\sf s}_{t_0}}: \widetilde{X}^{\phi \hskip 1pt {\sf s}_{t_0}}\rightarrow \Pp^1_{k}$ is $k$-isomorphic to the specia\-li\-za\-tion of the self-twisted cover $\widetilde{f\otimes_k k(U)}^{\phi_U}: \widetilde{X \otimes_k k(U)}^{\phi_U}\rightarrow \Pp^1_{k(U)}$   at $U=t_0$.
\end{lemma}

\begin{proof}
Set $d=|G|$. By construction, the extension $\widetilde{F{k(U)}}^{F_U}/k(U)(T)$ is regular over $k(U)$. From the Bertini-Noether theorem, for every $t_0\in k$ but in a finite subset ${\mathcal E}$, which we possibly enlarge to contain the branch point set ${\mathbf t}$, the specialized extension at $U=t_0$ is regular over $k$ and is of degree 
$$[\widetilde{F{k(U)}}^{F_U}:k(U)(T)]=[Fk(U):k(U)(T)] = [F:k(T)]=d.$$
\noindent
Up to enlarging again ${\mathcal E}$, one may also assume that the genus of this specialization 
is the same as the genus of the function field $\widetilde{F{k(U)}}$, which equals $g$, the genus of $F$. The rest of the proof shows that this specialization is the extension $\widetilde F^{F_{t_0}}/k(T)$.

Set $d=|G|$. Pick primitive elements ${\mathcal Y}$ and $\widetilde{\mathcal Y}_U$ of the two extensions $F/k(T)$ and $\widetilde{F{k(U)}}^{\phi_U}/k(U)(T)$, integral over $k[T]$ and $k[U,T]$ respectively. 
As $\widetilde{F{k(U)}}^{\phi_U} \subset Fk(U)$, one can write 
$$\widetilde{\mathcal Y}_U= \sum_{i=0}^{d-1} a_i(U) {\mathcal Y}^i$$
\noindent
with $a_0(U),\ldots,a_{d-1}(U) \in k(U)$. Enlarge the set ${\mathcal E}$ to contain the poles of $a_0(U),\ldots,a_{d-1}(U)$. Fix $t_0\in k\setminus {\mathcal E}$. Consider the corresponding specialization  $\widetilde{\mathcal Y}_{t_0} =  \sum_{i=0}^{d-1} a_i(t_0) {\mathcal Y}^i$. The associated extension $k(T, \widetilde{\mathcal Y}_{t_0})/k(T)$ is the specialization of $\widetilde{F{k(U)}}^{\phi_U}/k(U)(T)$ at $U=t_0$. By construction $\widetilde{\mathcal Y}_{t_0} \in F$. The last paragraph of the proof below shows that $\widetilde{\mathcal Y}_{t_0}$ is fixed by all elements in ${\sf s}_{t_0}(\Gabs_{k})$. We will then be able to conclude that $k(T, \widetilde{\mathcal Y}_{t_0}) \subset \widetilde F^{\phi_{t_0}}$ and finally that these two fields are equal since $[k(T, \widetilde{\mathcal Y}_{t_0}):k(T)] = [\widetilde F^{\phi_{t_0}}:k(T)]=d$.
\vskip 1mm
As $U\notin {\mathbf t}$, there exists an embedding 
$$\widetilde{F{k(U)}}^{\phi_U} \rightarrow \overline{k(U)} ((T-U))$$
\noindent
which maps ${\mathcal Y}_U$ to a formal power series 
$$\widetilde{\mathcal Y}_U = \sum_{n=0}^\infty b_n(U) (T-U)^n \hskip 3mm \hbox{with $b_n(U)\in \overline{k(U)}$ ($n\geq 0$)}.$$ 
\noindent
Furthermore, $\widetilde{F{k(U)}}^{\phi_U}$ is fixed by all elements ${\sf s}_U(\sigma) \in {\sf s}_U(\Gabs_{k(U)})$, which, by definition of ${\sf s}_U$, act  {\it via} the action of $\sigma \in \Gabs_{k(U)}$ on the coefficients $b_n(U)$; conclude that $b_n(U)\in {k(U)}$ ($n\geq 0$).
%
Finally from the Eisenstein theorem\footnote{This classical result is often stated for formal power series $\sum_{n\geq 0} b_n T^n$, algebraic over $\Qq(T)$ and with coefficients $b_n \in \overline \Qq$,  but is true in a bigger generality including the situation where $\Qq$ and $\Zz$ are respectively replaced by $k(U)$ and $k[U]$. For example, the proof given in \cite{dwork-robba} easily extends to this situation.},
there exists a polynomial $E(U) \in k[U]$ such that $E(U)^{n+1} \hskip 2pt b_n(U)\in k[U]$ for every $n\geq 0$. Enlarge again the set ${\mathcal E}$ to contain the roots of $E(U)$.
For $t_0 \in k\setminus {\mathcal E}$, specializing $U$ to $t_0$ in the displayed formula above produces $\widetilde{\mathcal Y}_{t_0}$ as a formal power series in $k[[T-t_0]]$, which amounts to saying that, up to some $k$-isomorphism,  $\widetilde{\mathcal Y}_{t_0}$ 
and so $\widetilde F^{\phi_{t_0}}$ are fixed by all elements in ${\sf s}_{t_0}(\Gabs_{k})$.
\end{proof}

Let $\widetilde P(U,T,Y) \in k[U,T,Y]$ be the irreducible polynomial of $\widetilde{\mathcal Y}_U$ over $k[U,T]$. Theorem \ref{prop:self-twist} holds for this polynomial $\widetilde P(U,T,Y)$ (up to enlarging again the finite set ${\mathcal E}$). When $k=\Qq$ we may 
and will choose the element $\widetilde{\mathcal Y}_U$ integral over $\Zz[T,Y]$ (and not just $\Qq(T,Y)$) so that
$\widetilde P(U,T,Y)$ lies in $\Zz[U,T,Y]$ and will assume further that the coefficients of $\widetilde P(U,T,Y)$ are relatively prime.

\section{Proof of theorem \ref{thm:main-plus}} \label{sec:proof}

We retain the notation and data introduced in \S \ref{ssec:basics}.

%


%
%
%
%
%
%
%

Fix a real number $x> p_0$ and a Frobenius data  ${\mathcal F}_x$ on ${S}_x$.

Fix also a subset $S_0$ of primes $p\in ]p_{-1},p_0]$, with as many elements as there are non trivial conjugacy classes in $G$. Associate then in a one-one way a non-trivial conjugacy class  ${\mathcal F}_p$ to each 
prime $p\in S_0$. Set $S_{0x} = S_0\cup S_x$ and denote the Frobenius data $({\mathcal F}_p)_{p\in S_{0x}}$ by ${\mathcal F}_{0x}$.


\subsection{First part: many good specializations $t_0\in \Zz$}  
The goal of the first part is proposition \ref{prop:first-part} which shows that there are ``many'' $t_0\in \Zz$ such that 
conditions (i)-(iv) of theorem \ref{thm:main-plus} are satisfied. The goal of the second part will be to show that there are ``many'' distinct corresponding extensions $F_{t_0}/\Qq$.

We use the method of \cite{DEGha} for this first part. We re-explain it in the special context of this paper and make the adjustments that we will need for the rest of the proof. We refer to \cite{DEGha} for more details 
on the main arguments and for references. 
Working over number fields and even over $\Qq$, we can give improved quantitative conclusions (compared to the existence statements of \cite{DEGha}).
As in \cite{DEGha}, there is first a local stage followed by a globalization argument.

\subsubsection{Local stage} Below, given $t_0\in \Qq_p$ we say that $t_0\notin {\mathbf t}$ modulo $p$ if $t_0$ does not meet any of the branch points of $F/\Qq(T)$ modulo $p$ \footnote{Recall that for two points $t,t^\prime \in \overline{\Qq_p} \cup \{\infty\}$, meeting modulo $p$ means that  $|t|_{\overline p} \le 1$, $|t^\prime|_{\overline p} \le 1$ and $|t-t^\prime|_{\overline p} < 1$, or else $|t|_{\overline p} \ge 1$, $|t^\prime|_{\overline p} \ge 1$ and $|t^{-1}-(t^\prime)^{-1}|_{\overline p} < 1$, where $\overline p$ is some arbitrary prolongation of the $p$-adic absolute value $v$ to $\overline{\Qq_p}$.}.

\begin{proposition} \label{prop:degha-method}
Given our regular $\Qq$-G-Galois cover $f:X\rightarrow \Pp^1$, a prime $p$ and a subset ${\mathcal F}_p\subset G$ consisting of a non-empty union of conjugacy classes of $G$, consider the set 
$${\mathcal T}({\mathcal F}_p) = \left\{t_0\in \Zz \hskip 2mm \left| 
\begin{matrix}
& \hbox{\rm $t_0\notin {\mathbf t}$ modulo $p$} \hfill \cr
& {\rm Frob}_p(F_{t_0}/\Qq) \in {\mathcal F_p} \hfill \cr
\end{matrix}\right.
\right\}.$$
\noindent
If $p$ is a good prime for $f$, the set ${\mathcal T}({\mathcal F}_p)$ is the union of  cosets modulo $p$. Furthermore,
the number $\nu({\mathcal F}_p)$ of these cosets satisfies
$$\nu({\mathcal F}_p) \geq \frac{|{\mathcal F}_p|}{|G|} \times (p+1-2g\sqrt{p} -|G|(r+1))$$

$$\hbox{\it and}\hskip 5mm \nu({\mathcal F}_p) \leq \frac{|{\mathcal F}_p|}{|G|} \times (p+1+2g\sqrt{p} ).$$
\end{proposition}

\begin{proof}
We follow the method from \cite{DEGha}. 
Similar estimates though not in this explicit form 
can also be found in \cite{ekedahl}.

We may and will assume that the subset ${\mathcal F}_p$ consists of a single conjugacy class.

Set $f_p=f\otimes_{\Qq}\Qq_p$ and denote the corresponding $\pi_1$-representation by $\phi_p: \pi_1(\Pp^1\setminus{\mathbf t},t)_{\Qq_p} \rightarrow G$. Pick an element $g_p\in {\mathcal F}_p$ and consider the unique unramified epimorphism $\varphi_p:\Gabs_{\Qq_p} \rightarrow \langle g_p \rangle$ that sends the Frobenius 
of $\Qq_p$ to $g_p$. 

The condition ``$t_0\notin {\mathbf t}$ modulo $p$'' implies that $p$ is unramified in the specialization $F_{t_0}/\Qq$.
Then $t_0\in {\mathcal T}({\mathcal F}_p)$ if and only if the specialization representation $\phi \circ{\sf s}_{t_0}: \Gabs_\Qq \rightarrow G$ of $F/\Qq(T)$ at $t_0$ is conjugate in $G$ to $\varphi_p:\Gabs_{\Qq_p} \rightarrow \langle g_p \rangle$. From the twisting lemma \ref{prop:twisted cover}, this is equivalent to the existence of 
a $k$-rational point above $t_0$ in the covering space of the twisted cover ${\widetilde{f_p}}^{\varphi_p}: {\widetilde{X_p}}^{\varphi_p} \rightarrow \Pp^1$. As $p$ is a good prime, this last cover has good reduction; denote the special fiber by 
${\widetilde{\sf f}}_p: {\widetilde{\sf X}}_p \rightarrow \Pp^1_{\Ff_p}$. The last existence condition is then equivalent to the existence of some point $\overline x\in {\widetilde{\sf X}}_p(\Ff_p)$ above the coset $\overline{t_0} \in \Pp^1(\Ff_p)$ of $t_0$ modulo $p$: the direct part is clear while the converse follows from Hensel's lemma.

From the Lang-Weil bound, the number of $\Ff_p$-rational points on ${\widetilde{\sf X}}_p$ is $\geq p+1 - 2g\sqrt{p}$. Removing the points that lie above a branch point or the point at infinity leads to the announced first estimate, a final observation for this calculation being that for $t_0 \notin {\mathbf t}$ modulo $p$, the number of $\Ff_p$-rational points 
$\overline x\in {\widetilde{\sf X}}_p(\Ff_p)$ above $\overline{t_0}$ is $|{\rm Cen}_G(g_p)| = |G|/|{\mathcal F}_p|$: this number is indeed the same as the number of $\omega\in G$ such that  $\phi \circ{\sf s}_{t_0} = \omega \varphi_p \omega^{-1}$ (as the proof of the twisting lemma in \cite{DEGha} shows).
Using the upper bound part of Lang-Weil leads to the second estimate.
\end{proof}

If in addition $p\geq r^2 |G|^2$ (in particular if $p\in S_{0x}$), then the right-hand side term in the inequality of proposition \ref{prop:degha-method} is $>0$ (use that $g < r|G|/2 -1$ if $|G|>1$, which follows from Riemann-Hurwitz).


\begin{proposition} \label{prop:francois}
Assume the branch point $t_1$ of the $\Qq$-G-Galois cover $f:X\rightarrow \Pp^1$ is in $\Zz$. Given a prime $p$, consider the set 
$${\mathcal T}(\hbox{\rm ra}/p) = \left\{t_0\in \Zz \hskip 1mm \left| \hskip 1mm F_{t_0}/\Qq\ \hbox{\rm is ramified at}\ p \right.
\right\}.$$
\noindent
If $p$ is a good prime for $f$, the set ${\mathcal T}(\hbox{\rm ra}/p)$ contains the coset of  $t_1+p\in \Zz$ modulo $p^2$.
\end{proposition}

\begin{proof} Let $t_0\in \Zz$ such that $t_0 \equiv t_1+p$ modulo $p^2$. Then $t_0-t_1$ is of $p$-adic valuation $1$. 
As $p$ is good, it follows that $F_{t_0}/\Qq$ is ramified at $p$. This last conclusion is part of the ``Grothendieck-Beckmann theorem'' for which we refer to \cite{SGA} and \cite[proposition 4.2]{Beckmann}; see also \cite{Legr1} where 
this result is discussed together with further developments in the spirit of proposition \ref{prop:francois}.
\end{proof}

\subsubsection{Globalization}
Set
$$\left\{\begin{matrix}
& \displaystyle \beta = \Pi({S_0}) \hfill \cr
&  \displaystyle B(x)= \beta \hskip 2pt \Pi({S})^2 \hskip 1pt \Pi({S_x}) \hfill \cr
\end{matrix} \right.$$

\noindent
and consider the intersection
$$\bigcap_{p\in S_{0x}} {\mathcal T}({\mathcal F}_p) \cap \bigcap_{p\in S} {\mathcal T}(\hbox{\rm ra}/p).$$ 

\noindent
From proposition \ref{prop:degha-method}, proposition \ref{prop:francois} and the Chinese remainder theorem, this set contains 
$${\mathcal N}(S,{\mathcal F}_{0x})=\prod_{p\in S_{0x}} \nu({\mathcal F}_p)$$
\noindent
\hbox{\rm cosets modulo} $B(x)$.
Denote the set of their representatives in $[1,B(x)]$ by ${\mathcal T}(S,{\mathcal F}_{0x})$; the cardinality of this set is  ${\mathcal N}(S,{\mathcal F}_{0x})$. 

%

\begin{proposition} \label{prop:first-part} {\rm (a)} For every integer $t_0\in {\mathcal T}(S,{\mathcal F}_{0x})$, the extension $F_{t_0}/\Qq$ satisfies the four conditions {\rm (i)-(iv)} from theorem \ref{thm:main-plus}, with {\rm (ii)} even replaced by the following sharper version {\rm (ii+)} of {\rm (ii)}, that is 
%
\vskip 1mm
\noindent
{\rm (i)} $\Gal(F_{t_0}/\Qq)=G$,
\vskip 1mm

\noindent
{\rm (ii+)} $F_{t_0}/\Qq$ is unramified and ${\rm Frob}_p(F_{t_0}/\Qq) \in {\mathcal F_p}$ for every $p\in S_{0x}$
{\rm (and not just for every $p\in S_x$)},
\vskip 1mm

\noindent
{\rm (iii)} $F_{t_0}/\Qq$ is ramified at $p$ for every $p\in S$,
\vskip 1mm

\noindent
{\rm (iv)} $| d_{F_{t_0}} | \leq \rho(x)$.
\vskip -4mm

$$ \hbox{\it We have} \hskip 5mm {\mathcal N}(S,{\mathcal F}_{0x}) \geq \hskip 2pt \chi({\mathcal F}_x)\hskip 1mm  \frac{B(x)}{\beta \hskip 1pt \Pi({S})^2} \left( \frac{1}{2r|G|}\right)^{|S_x|} \hskip 22mm\leqno(\hbox{\rm b})$$
\vskip 2mm

\noindent
{\rm (c)} The set  of integers $t_0\in \Zz$ such that for each $p\in S_x$, $F_{t_0}/\Qq$ is unramified and ${\rm Frob}_p(F_{t_0}/\Qq) \in {\mathcal F_p}$ consists of cosets modulo $\Pi(S_x)$ and the set ${\mathcal T}(\emptyset,{\mathcal F}_x)$ of their representatives in $[1,\Pi({S_x})]$ is of cardinality
$${\mathcal N}(\emptyset,{\mathcal F}_x) = \prod_{p\in S_x} \nu({\mathcal F}_p) \leq \hskip 2pt \chi({\mathcal F}_x) \hskip 1mm \frac{\Pi(S_x)}{\beta} \hskip 1mm  \left( 2 - \lambda\right)^{|S_x|}$$
\noindent
\hskip 0mm where $\lambda = (r|G|-1)/r^2 |G|^2$.
\end{proposition}

Conclusion (c) proves conclusion (d) of theorem \ref{thm:main-plus}.

\begin{proof} (a) Fix $t_0\in {\mathcal T}(S,{\mathcal F}_{0x})$ (or more generally congruent modulo $B(x)$ to some element in ${\mathcal T}(S,{\mathcal F}_{0x})$). 

Conditions (ii+), (iii) hold by definition of the sets ${\mathcal T}({\mathcal F}_p)$ and ${\mathcal T}(\hbox{\rm ra}/p)$. 

A classical argument then shows that (i) follows from (ii+): indeed because of the Frobenius condition at the primes $p\in S_0$, the subgroup $\Gal(F_{t_0}/\Qq)\subset G$ meets every conjugacy class of $G$; from a lemma of Jordan \cite{jordan}, it is all of $G$.

From (i), the polynomial $P(t_0,Y)$ is irreducible in $\Qq[Y]$. As it is monic and with integral coefficients, its discriminant, which is $\Delta_P(t_0)$, is a multiple of the absolute discriminant $d_{F_{t_0}}$ of the extension $F_{t_0}/\Qq$. Conjoined with $1\leq t_0 \leq B(x)$, this leads to
$$ |d_{F_{t_0}}| \leq (1+\delta_P) \hskip 2pt H(\Delta_P)  \hskip 2pt B(x)^{\delta_P}$$ 
\noindent
and conclusion (iv) follows from the definition of $\rho(x)$ (given in \S \ref{ssec:data}).
\smallskip

(b) Using proposition \ref{prop:degha-method}, we obtain
$$ \begin{matrix}
{\mathcal N}(S,{\mathcal F}_{0x}) & \displaystyle \geq \prod_{y\in S_x}\frac{|{\mathcal F}_p|}{|G|} \times (p+1-2g\sqrt{p} -|G|(r+1))\hfill \cr
 &\displaystyle \geq \chi ({\mathcal F}_x) \times  \prod_{p\in S_x} p  \times \prod_{p\in S_x} \left( 1+ \frac{1}{p} -\frac{2g}{\sqrt{p} }- \frac{(r+1)|G|}{p}\right) \hfill \cr
\end{matrix}$$
\noindent
Using again that $g < r|G|/2 -1$ (if $|G|>1$) and that $p\geq  r^2 |G|^2$ for each $p\in S_x$, we have
$$\begin{matrix}
\displaystyle 1+ \frac{1}{p} -\frac{2g}{\sqrt{p} }- \frac{|G|(r+1)}{p} & \displaystyle > \hskip 2mm 1 - \frac{r|G|-2}{r|G|} - \frac{(r+1)|G|}{r^2 |G|^2} \hfill\cr
 & \displaystyle = \hskip 2mm \frac{2}{r|G|} - \frac{(r+1)|G|}{r^2 |G|^2} \hfill \cr
  & \displaystyle = \hskip 2mm \frac{(r-1)|G|}{r^2 |G|^2} \hskip 2mm \geq \hskip 2mm \frac{1}{2r|G|}\hfill \cr
\end{matrix}$$
\noindent
which yields the announced first estimate. 
\smallskip

(c) Here we use the conclusion from proposition \ref{prop:degha-method} that for each $p\in S_x$, the set ${\mathcal T}({\mathcal F}_p)$ consists exactly of $\nu({\mathcal F}_p)$ cosets modulo $p$. Combined with the Chinese remainder, this gives that the set ${\mathcal T}(\emptyset,{\mathcal F}_x)$ consists of exactly ${\mathcal N}(\emptyset,{\mathcal F}_x) = \prod_{p\in S_x} \nu({\mathcal F}_p)$ elements. Proceed then similarly as in (b) but using the upper bound part of proposition \ref{prop:degha-method} to obtain the desired estimate.
\end{proof}

\begin{remark} Consider the situation with $S = \emptyset$ and allowing no local condition at some primes $p\in S_x$ --- no restriction on ${\rm Frob}_p(F_{t_0}/\Qq)$ and no unramified condition ---. We have $\nu({\mathcal F}_p) = p$ for such primes and obtain
this generalized lower bound: if $S_x^\prime\subset S_x$ is the subset of primes where there {\it is} a local condition, then
$${\mathcal N}(\emptyset,{\mathcal F}_{0x}) \geq \hskip 2pt \chi({\mathcal F}_x)\hskip 1mm  \frac{B(x)}{\beta} \left( \frac{1}{2r|G|}\right)^{|S_x^{\prime}|} .$$
\noindent 
In particular, the number of integers $t_0 \in [1,B(x)]$ where conditions {\rm (i)} $\Gal(F_{t_0}/\Qq)=G$ and 
{\rm (iv)} $| d_{F_{t_0}} | \leq \rho(x)$ hold ({\it i.e.} no local condition at any prime) is $ \geq \hskip 2pt  {B(x)}/{\beta} $.
%
\end{remark}

\subsection{Second part: many good specializations $F_{t_0}/\Qq$} 

\subsubsection{Reduction to counting integral points on curves}
We will now estimate the number, say ${N}(S,{\mathcal F}_{0x})$, of distinct specializations $F_{t_0}/\Qq$ with $t_0\in {\mathcal T}(S,{\mathcal F}_{0x})$. We will give a lower bound for the number of non conjugate specialization representations $\phi \circ {\sf s}_{t_0}:\Gabs_{\Qq} \rightarrow G$ with $t_0\in {\mathcal T}(S,{\mathcal F}_{0x})$. Given two such representations, the associated field extensions are equal if and only if the representations have the same kernel, or, equivalently, if they differ by some automorphism of $G$. Dividing the previous bound by $|{\rm Aut}(G)|$ will thus yield the desired bound for ${N}(S,{\mathcal F}_{0x})$.

%

Consider the polynomial $\widetilde P(U,T,Y) \in \Zz[U,T,Y]$ given by theorem \ref{prop:self-twist} and its discriminant $\Delta_{\widetilde P} \in \Zz[U,T]$ (relative to $Y$). As  $\widetilde P(U,T,Y)$ is irreducible in $\Qq(U,T)[Y]$, $\Delta_{\widetilde P}(U,T)\not= 0$. Write it as a polynomial in $T$ of degree $N$ and denote its leading coefficient by $\Delta_{\widetilde P,N} (U)$; we have $\Delta_{\widetilde P,N} (U)\in \Zz[U]$ and $\Delta_{\widetilde P,N} (U)\not= 0$.

Drop from the set ${\mathcal T}(S,{\mathcal F}_{0x})$ the finitely many integers $u_0$ for which $\widetilde \Delta_{\widetilde P,N} (u_0)=0$ or which are in the exceptional set ${\mathcal E}$ from theorem \ref{prop:self-twist}. 
Denote the resulting set by ${\mathcal T}(S,{\mathcal F}_{0x})^\prime$ and the number of dropped elements by ${\rm E}$. 

Fix $u_0\in {\mathcal T}(S,{\mathcal F}_{0x})^\prime$ and consider the fiber-twisted cover at $u_0$:
$$\widetilde{f}^{\phi \hskip 1pt {\sf s}_{u_0}}: \widetilde{X}^{\phi \hskip 1pt {\sf s}_{u_0}}\rightarrow \Pp^1_{\Qq}.$$
\noindent
Let $t_0\in {\mathcal T}(S,{\mathcal F}_{0x})^\prime$. From the twisting lemma \ref{prop:twisted cover}, the two representations $\phi \circ {\sf s}_{u_0}$ and $\phi \circ {\sf s}_{t_0}$ are conjugate in $G$ if and only if there exists $x_0\in \widetilde{X}^{\phi \hskip 1pt {\sf s}_{u_0}}(\Qq)$ such that $\widetilde{f}^{\phi \hskip 1pt {\sf s}_{u_0}}(x_0)=t_0$.

We have $\Delta_{\widetilde P}(u_0,t_0) \not=0$ except for at most $N$ integers $t_0$. For the non-exceptional $t_0$, the polynomial $\widetilde P(u_0,t_0,Y)$ has only distinct roots $y \in \overline \Qq$ and, using theorem \ref{prop:self-twist}, the corresponding points $(t_0,y)$ 
on the affine curve $\widetilde P(u_0,t,y)=0$ exactly correspond to the points $x$ on the smooth projective curve $\widetilde{X}^{\phi \hskip 1pt {\sf s}_{u_0}}$ above $t_0$. Furthermore, in this correspondence, the $\Qq$-rational points $x$ correspond to the couples $(t_0,y)$ with $y\in \Qq$. Conclude that up to some term $\leq N$, the number of $t_0$ for which $\phi \circ {\sf s}_{u_0}$ and $\phi \circ {\sf s}_{t_0}$ are conjugate in $G$ is equal to the number of $\Qq$-rational points $(t_0,y)$ on the affine curve $\widetilde P(u_0,t,y)=0$. 

Note further that such a $\Qq$-rational point $(t_0,y)$ has necessarily integral coordinates as $t_0\in \Zz$ and $\widetilde P(u_0,T,Y)\in \Zz[T,Y]$ and is monic in $Y$.  Therefore we are reduced to counting the integers $t_0\in [1,B(x)]$ such that there is an integral point $(t_0,y)\in \Zz^2$ on the curve $\widetilde P(u_0,t,y)=0$.

\subsubsection{Diophantine estimates}
The constants $c_i$, $i>0$ that will enter depend only on the polynomial $P(T,Y)$.

The curve $\widetilde P(u_0,t,y)=0$ is of genus $g$ (theorem \ref{prop:self-twist} (c)) and we have
$$\left\{ 
\begin{matrix}
& \deg(\widetilde P(u_0,T,Y)) \leq \deg(\widetilde P(U,T,Y)) = c_1 \hfill \cr
& \deg_Y(\widetilde P(u_0,T,Y)) = \deg_Y (\widetilde P(U,T,Y)) = |G| \hfill \cr
& H(\widetilde P(u_0,T,Y)) \leq c_2 u_0^{c_3} \leq c_2 B(x)^{c_3} \hfill \cr
\end{matrix}
\right.$$

\noindent
For real numbers $\gamma ,D, H, B \geq 0$ and $d_Y\geq 2$, consider all polynomials $F\in \Zz[T,Y]$, primitive, 
monic in $Y$, irreducible in $\overline \Qq(T)[Y]$, such that $\deg_Y(F) = d_Y$, of total degree $\leq D$, of height $\leq H$ 
and such that the affine curve $P(t,y)=0$ is of genus $\leq \gamma$. For each such polynomial, the number of integers $t\in [1,B]$ such that there exists $y\in \Zz$ such that $F(t,y)=0$ is a finite set. Denote by $Z(\gamma,D, d_Y, H, B)$  the maximal cardinality of all these finite sets.

Using the diophantine parameter $Z(\gamma,D, d_Y, H, B)$, conclude that the number of $t_0\in {\mathcal T}(S,{\mathcal F}_{0x})^\prime$ such that the two representations $\phi \circ {\sf s}_{u_0}$ and $\phi \circ {\sf s}_{t_0}$ are conjugate in $G$ is less than or equal to
$$Z(g,c_1,|G|,c_2 B(x)^{c_3},B(x)).$$

Thus we obtain
$${N}(S,{\mathcal F}_{0x}) \geq \frac{{\mathcal N}(S,{\mathcal F}_{0x})-\hbox{\rm E}}{|{\rm Aut}(G)| \hskip 3pt [Z(g,c_1,|G|, c_2 B(x)^{c_3},B(x))+N]}$$

\noindent
Next take into account proposition \ref{prop:first-part} and note that $N_F(x,S,{\mathcal F}_x) \geq {N}(S,{\mathcal F}_{0x})$ to write
$$N_F(x,S,{\mathcal F}_x)  \geq \frac{\displaystyle \frac{\chi({\mathcal F}_x) \hskip 1mm (2r|G|)^{-|S_x|}\hskip 1mm  B(x)}{\beta \hskip 2pt \Pi({S)}^{2}}-\hbox{\rm E}}{|{\rm Aut}(G)| \hskip 3pt [Z(g,c_1,|G|,c_2 B(x)^{c_3},B(x))+N]}$$

Assume that the genus $g$ of $X$ is $\geq 2$ and that Lang's conjecture holds. This conjecture is that if $V$ is a variety of general type defined over a number field $K$ then the set $V(K)$ of $K$-rational points is not Zariski-dense in $V$ \cite{lang}. We will use it through the following consequence proved by Caporaso, Harris and Mazur \cite{charm}: they  showed that Lang's conjecture implies that for every number field $K$ and every integer $g\geq 2$ there exists a finite integer $B(g,K)$ such that ${\rm card} (C(K)) \leq B(g,K)$ for every curve $C$ of genus $g$ defined over $K$.

Under this conjecture we obtain
$$Z(g,c_1,|G|,c_2B(x)^{c_3},B(x)) + N \leq  c_4.$$

In the general case $g\geq 0$ we use an unconditional result of Walkowiak \cite[\S 2.4]{Wa} which shows that
if $d_Y \geq 2$ then
$$Z(\gamma,D, d_Y, H, B) \leq a_1 D^{a_2} (\log H^+)^{a_3} B^{1/d_Y} (\log B)^{a_4}$$
\noindent
where $H^+ = \max(H,e^e)$ and $a_1, \ldots, a_4$ are absolute constants. See \S \ref{ssec:walkowiak} for more on this result.
We deduce
$$Z(g,c_1,|G|,c_2B(x)^{c_3},B(x)) + N \leq  c_5 B(x)^{1/|G|} \log(B(x)^{c_{6}}.$$
\noindent
Conclude that unconditionally:
$$N_F(x,S,{\mathcal F}_x) \geq  \hskip 2pt c_7 \hskip 2pt \frac{\chi({\mathcal F}_x)}{\Pi({S})^{2}} \hskip 4pt \frac{B(x)^{1-1/|G|}}{(\log B(x))^{c_9} \hskip 2pt c_8^{|S_x|}} - c_{10}$$

\noindent
and, under Lang's conjecture:
$$N_F(x,S,{\mathcal F}_x) \geq  \hskip 2pt c_7 \hskip 2pt \frac{\chi({\mathcal F}_x)}{\Pi({S})^{2}} \hskip 4pt \frac{B(x)}{c_8^{|S_x|}} - c_{10}.$$

\noindent
Note that $c_{11} \Pi(x)\leq B(x) \leq \Pi(x) \Pi(S)^2$, that $|S_x| \leq \pi(x)$ and $0<c_8<1$ to obtain the estimates of theorem \ref{thm:main-plus} (a) and (c). Theorem \ref{thm:main-plus} (b) follows from the containments ${\mathcal T}(S,{\mathcal F}_{0x}) \subset [1,B(x)] \subset [1,\Pi(S)\hskip 1pt \Pi (x)]$.

\subsubsection{Walkowiak's result} \label{ssec:walkowiak} Let $F\in \Zz[T,Y]$ be a polynomial, irreducible in $\Zz[T,Y]$. Set $D=\deg(F)$ and  $H^+=\max(H(F),e^e)$. The result we use in the proof above is the following.

\begin{theorem}[Walkowiak]\label{thm:walkowiak} Assume $\deg_YF \geq 2$. There exist absolute constants $a_1, \ldots, a_4$ such that for every real number $B>0$, the number of integers $t_0\in [1,B]$ such that $F(t_0,Y)$ has a root in $\Zz$ is less than   
$$a_1 D^{a_2} (\log H^+)^{a_3} B^{1/\deg_Y(F)} (\log B)^{a_4}.$$
\end{theorem}

\noindent
This result is proved in \cite[\S 2.4]{Wa} but with  $B^{1/2}$ instead of $B^{1/\deg_Y(F)}$. We explain here how to modify Walkowiak's arguments to obtain the better exponent $1/\deg_Y(F)$. The only change to make is in the final stage of 
the proof in \cite[\S 2.2-2.3]{Wa}.

\begin{proof}
Walkowiak's central result is the following bound for the number $N(F,B)$ of $(t,y)\in \Zz^2$ such that $\max(|t|,|y|)\leq B$ and $F(t,y)=0$:
$$ N(F,B) \leq 2^{36} D^5 \log^3(1250d^{11} B^{5D-1}) \log^2(B)\hskip 2pt B^{1/D}.$$
\noindent
To prove theorem \ref{thm:walkowiak}, his basic idea is  to use Liouville's inequality to get upper bounds $|y|\leq B^\prime$ for roots $y\in \Zz$ of polynomials $F(t_0,Y)$ with $t_0\in [1,B]$; the bound above for $N(F,B)$ with $B$ taken to be $B^\prime$ provides then a bound for the desired set. The main terms that appear in the resulting bound come from $(B^\prime)^{1/D}$. They may be too big however in some cases and Walkowiak uses a trick to obtain his final bound in $B^{1/2}$. In order to obtain $B^{1/n}$ instead, Walkowiak's trick should be modified as follows.

%

Set $L_1=\log(H^+)$, $L_2=\log(\log(H^+))$, $m=\deg_TF$ and $n=\deg_YF$; one may assume $m>0$.
Let $t_0\in [1,B]$ such that $F(t_0,Y)$ has a root $y\in\Zz$. Liouville's inequality gives
$$|y| \leq 2(m+1) H^+ B^m = B^\prime.$$
\vskip 1mm
\noindent
The main terms in $(B^\prime)^{1/D}$ 
are $(H^+)^{1/D}$ and $(B^m)^{1/D}$. 

\vskip 1mm
\noindent
{\it Case 1}: $ mnL_1/L_2\leq D$. On the one hand, we have $1/D \leq L_2/L_1$ and so $(H^+)^{1/D} \leq (H^+)^{L_2/L_1}=\log (H^+)$. On the other hand $m/D \leq 1/n$ and so $B^{m/D}\leq B^{1/n}$. The upper bound for $N(F,B^\prime)$ is indeed as announced in the statement of theorem \ref{thm:walkowiak}.

\vskip 2mm
\noindent
{\it Case 2}: $ mnL_1/L_2> D$. Set $E=[mnL_1/L_2]+1$ and consider the polynomial $G\in \Zz[T,Y]$ defined by
$G(T,Y)=F(T,T^E+Y)$. For $y^\prime = y-t_0^E$ we have $G(t_0,y^\prime) = 0$ and 
$$|y^\prime| \leq 2(m+1) H^+ B^m + B^E \leq 2(m+1) H^+ B^E = B''.$$
\noindent
Use then the upper bound for $N(F,B)$ with $F$ and $B$ respectively taken to be $G$ and $B''$.
As $\deg_YG=\deg_YF=n$ and $nE\leq \deg G\leq nE+m$, the main terms are in this case
$$(H^+)^{1/\deg(G)} \leq (H^+)^{1/nE} \leq   (H^+)^{L_1/L_2} = \log(H^+)$$
$$\hbox{\rm and} \hskip 5mm B^{E/\deg G} \leq B^{1/n}.$$

\noindent
Again the upper bound for $N(G,B'')$ is as announced. \end{proof}

\bibliography{LocalMalle}
\bibliographystyle{alpha}

\end{document}